\documentclass[11pt,twoside]{article}
\usepackage{amsmath,amsfonts,amsthm}

\newtheorem{thm}{Theorem}[section]
\newtheorem{prop}[thm]{Proposition}
\newtheorem{lem}[thm]{Lemma}
\newtheorem{cor}[thm]{Corollary}

\theoremstyle{definition}
\newtheorem{defin}[thm]{Definition}

\theoremstyle{remark}
\newtheorem{rem}[thm]{Remark}
\newtheorem{ex}[thm]{Example}

\numberwithin{equation}{section}
\newcommand{\x}{\times}
\newcommand{\ox}{\otimes}
\def\lacute{\mathopen{<}}
\def\racute{\mathopen{>}}
\newcommand{\scal}[2]{{\lacute#1,#2\racute}}
\newcommand{\ensemble}[2]{\{\,#1\mid#2\,\}}
\newcommand{\C}{{\mathbb C}}
\newcommand{\R}{{\mathbb R}}
\newcommand{\Z}{{\mathbb Z}}

\newcommand{\CaD}{\mathcal D} 
\newcommand{\CE}{\mathcal E}       
\newcommand{\CF}{\mathcal F}
\newcommand{\CH}{\mathcal H}
\newcommand{\CI}{\mathcal I}
\newcommand{\CJ}{\mathcal J}
\newcommand{\CM}{\mathcal M}
\newcommand{\CN}{\mathcal N}
\newcommand{\CO}{\mathcal O}

\newcommand{\vespa}{\vspace{1em}}
\newcommand{\dsur}[1]{\frac \partial{\partial#1} }               %  del/del x
\newcommand{\TX}{{T^*X}}
\newcommand{\TY}{{T^*Y}}
\newcommand{\TZ}{{T^*Z}}
\newcommand{\TL}{{T^*\Lambda}}
\newcommand{\TYX}{{T^*_YX}}
\newcommand{\TZX}{{T^*_ZX}}
\newcommand{\TZY}{{T^*_ZY}}
\newcommand{\dTX}{{\dot T^*X}}
\newcommand{\OX}{{{\mathcal O_X}}}
\newcommand{\OY}{{\mathcal O_Y}}
\newcommand{\cOXY}{{\mathcal O_{\widehat{X|Y}}}}
\newcommand{\DX}{{\mathcal D_X}}
\newcommand{\EX}{{{\mathcal E_X}}}
\newcommand{\EIX}{{{\mathcal E^\infty_X}}}
\newcommand{\EY}{{{\mathcal E_Y}}}
\newcommand{\EZ}{{{\mathcal E_Z}}}
\newcommand{\DZ}{{{\mathcal D_Z}}}
\newcommand{\EYX}{{\mathcal E_{Y\rightarrow X}}}
\newcommand{\BYX}{{{\mathcal B_{Y|X}}}}
\newcommand{\BIOX}{{{\mathcal B^\infty_{\{0\}|X}}}}
\newcommand{\BOX}{{{\mathcal B_{\{0\}|X}}}}
\newcommand{\BIYX}{{{\mathcal B^\infty_{Y|X}}}}
\newcommand{\BYfiX}{{{\mathcal B_{Y_\gf|X}}}}
\newcommand{\BIYfiX}{{{\mathcal B^\infty_{Y_\gf|X}}}}
\newcommand{\CYX}{{{\mathcal C_{Y|X}}}}
\newcommand{\CZY}{{{\mathcal C_{Z|Y}}}}
\newcommand{\CIOX}{{{\mathcal C^\infty_{\{0\}|X}}}}
\newcommand{\COX}{{{\mathcal C_{\{0\}|X}}}}
\newcommand{\CIYX}{{{\mathcal C^\infty_{Y|X}}}}
\newcommand{\CIZY}{{{\mathcal C^\infty_{Z|Y}}}}
\newcommand{\CZX}{{{\mathcal C_{Z|X}}}}
\newcommand{\CIZX}{{{\mathcal C^\infty_{Z|X}}}}

\newcommand{\CHM}{{Ch(\CM)}}
\newcommand{\CHMrs}{{Ch^2_\gL{\scriptstyle(r,s)}(\CM)}}
\DeclareMathOperator{\Ext}{\mathcal Ext}  
\DeclareMathOperator{\End}{\mathcal End}
\DeclareMathOperator{\RHOM}{\R\mathcal Hom}

\newcommand{\ga}{\alpha}
\newcommand{\gd}{\delta}
\newcommand{\gh}{\eta}
\newcommand{\gth}{\theta}
\newcommand{\gvt}{\vartheta}
\newcommand{\gl}{\lambda}
\newcommand{\gL}{\Lambda}
\newcommand{\gm}{\mu}
\newcommand{\gx}{\xi}
\newcommand{\gp}{\pi}
\newcommand{\gro}{\varrho}
\newcommand{\gs}{\sigma}
\newcommand{\gS}{\Sigma}
\newcommand{\gt}{\tau}
\newcommand{\gf}{\varphi}
\newcommand{\gF}{\Phi}
\newcommand{\gy}{\psi}             

\newcommand{\go}{\omega}           
\newcommand{\gO}{\Omega}

\begin{document}

\title{Regularity of a D-module along a submanifold}
\author{Yves Laurent}

\maketitle

\section*{Abstract}
We study how regularity along a submanifold of a differential or microdifferential
system can propagate from a family of submanifolds to another. The first result
is that a microdifferential system regular along a lagrangian foliation is regular. However,
when restricted to a fixed submanifold the corresponding result is true only
under a condition on the characteristic variety.

{\def\thefootnote{\relax}
\footnote{\hskip-0.6cm
{\it 1991 Mathematics Subject Classification.\/}~: 35A27.\newline
{\it Key words\/}~: D-modules, regularity, microlocal.
}} 

\thispagestyle{empty}
\newpage

\section*{Introduction}

This initial idea of this paper is a question M. Granger relative
to a paper with F. Castro-Jimenez \cite{GRANGER}:
If a holonomic $\DX$-module
is regular along a family of hyperplanes crossing on a linear subvariety, is it regular along
the intersection. 

If we consider this problem from the microlocal point of view,
it is equivalent to the following: if a holonomic $\DX$-module
is regular along the points of a submanifold, is it regular along
the submanifold itself. From this point of view, the problem is similar to the well-known theorem
of Kashiwara-Kawa{\"\i} \cite{KKREG}: if a $\DX$-module 
is regular along the points of an open set, it is regular on this set.

To solve the problem
we first prove a microlocal version of the result of Kashiwara-Kawa{\"\i}:

Consider a conic lagrangian foliation of an open set of the cotangent bundle 
to a complex manifold. Let $\CM$ 
be a holonomic microdifferential module defined on the open set. If $\CM$ is regular
along each leaf of the foliation then the module $\CM$ is regular.

Then we may give an answer to the initial problem. A natural framework
to state it is a "maximally degenerated" involutive submanifold
of the cotangent bundle. Such a variety carry a canonical conic lagrangian foliation
and contains also a canonical conic lagrangian submanifold, its degeneracy locus.
Then, the result is that if a holonomic microdifferential module is regular along the leaves of the foliation,
it is regular along the degeneracy locus under the condition that its characteristic variety
is contained in the maximally degenerated involutive submanifold.

This apply to our initial problem. If a holonomic differential or microdifferential module 
has its characteristic variety contained in the set $\scal x\gx=0$ of $T^*\C^n$ and regular along the hyperplanes containing the 
origin, then it is regular along the origin. In the same way, a holonomic module 
which is regular along the points of a submanifold $Y$ of a complex variety $X$
is regular along $Y$ under the condition that, in a neighborhood of the conormal to $Y$,
the characteristic variety is contained in the inverse image of $Y$ by the projection $\TX\to X$.

However, this result is not always true if the condition on the characteristic variety
is not satisfied. In the last example, it may be untrue if the singular support
of the module has components tangent to the variety $Y$. We show this by constructing
a counterexample.

In sections 1 and 2, we recall the different definitions of regularity and 
give some classical results that we will use later. 

In section 3, we prove a complex microlocal Cauchy theorem that we use in section
4 to prove our main result. In section 5, we show how this applies to 
maximally degenerated involutive manifolds and give examples.

Section 6, is devoted to the calculation of a counterexample when the
condition on the characteristic variety is not fulfilled.

\section{Regularity}

In dimension $1$ regularity of $\DX$-modules is equivalent to the
notion of differential equation with regular singularity. In higher
dimension, there are two different kind of regularity. The first one
is global, that is concerns a $\DX$-module on an open set while the
second is relative to a subvariety. In both cases, there is an equivalence
between growth conditions on the solutions and algebraic conditions
on the module itself. 

Algebraic conditions uses the so-called V-filtration which is a natural
extension of the algebraic conditions in the definition of "regular singular".

Let $X$ be a complex manifold and $Y$ a submanifold of $X$. Let $\OX$ be
the sheaf of holomorphic functions on $X$, $\DX$ the sheaf of differential operators
on $X$ with coefficients in $\OX$.

The sheaf $\DX$ is provided with the usual filtration by the order
of operators, this filtration will be denoted by $(\CaD_{X,m})_{m\ge 0})$.
Kashiwara defined in \cite{KVAN} an other filtration, the $V$-filtration, by:
\begin{equation}
V_k\DX=\ensemble{P\in\DX}{\forall \ell\in \Z, P\CI_Y^\ell\subset \CI_Y^{\ell-k}}
\end{equation}
where $\CI_Y$ is the ideal of definition of $Y$ and $\CI_Y^\ell=\OX$ if $\ell\le 0$.

If $Y$ is given in local coordinates by $Y=\ensemble{(x_1,\dots,x_p,t_1,\dots,t_q)}{t=0}$,
then the function $x_i$ and the derivations $D_{x_i}=\dsur{x_i}$ are of order $0$ for the V-filtration
while $t_j$ is of order $-1$ and $D_{t_j}$ of order $1$.

\vespa
Let $\gt:T_YX\to Y$ be the normal bundle to $Y$ in $X$ and $\CO_{[T_YX]}$ the sheaf of holomorphic
functions on $T_YX$ which are polynomial in the fibers of $\gt$. Let $\CO_{[T_YX]}[k]$ be the subsheaf
of $\CO_{[T_YX]}$ of homogeneous functions of degree $k$ in the fibers of $\gt$. There are canonical isomorphisms
between $\CI_Y^k/\CI_Y^{k-1}$ and $\gt_*\CO_{[T_YX]}[k]$, between $\bigoplus\CI_Y^k/\CI_Y^{k-1}$
and $\gt_*\CO_{[T_YX]}$. Hence 
the graded ring $gr^V\DX$ associated to the V-filtration on $\DX$ acts naturally 
on $\CO_{[T_YX]}$. An easy calculation \cite{SCHAPBOOK} shows that as a subring of $\End(\gt_*\CO_{[T_YX]})$ it is identified to 
$\gt_*\CaD_{[T_YX]}$ the sheaf of differential
operators on $T_YX$ with coefficients in $\CO_{[T_YX]}$ .

The Euler vector field $\gth$ of $T_YX$ is the vector field which acts on  $\CO_{[T_YX]}[k]$
by multiplication by $k$. Let $\gvt$ be any differential operator in $V_0\DX$ whose image in $gr^V_0\DX$
 is~$\gth$.

\begin{defin}\label{def:first}
The holonomic $\DX$-module $\CM$ is \textsl{regular along} $Y$ if any section $u$
of $\CM$ is annihilated by a differential operator of the form $\gvt^N + P + Q$
where $P$ is in $\CaD_{X,N-1}\cap V_{0}\DX$
and $Q$ is in $\CaD_{X,N}\cap V_{-1}\DX$
\end{defin}

\begin{defin}\label{def:regbf}
A polynomial $b$ is {\sl a regular b-function} for $u$ along $Y$ if there exists an operator $Q$ in 
$V_{-1}\DX\cap\CaD_{X,m}$ where $m$ is the degree of $b$ such that $(b(\gvt)+Q)u=0$. 
\end{defin}

It is proved in \cite{BFUNCT} that $\CM$ is regular along $Y$ if and only if all sections
of $u$ admit a regular $b$-function.

Let us denote by $\cOXY$ the formal completion of $\OX$ along $Y$, that is
$$\cOXY=\projlim k {\OX/\CI_Y^k}$$

We proved in \cite{INV} that if $\CM$ is regular along $Y$ then 
\begin{equation}
\label{equ:byx}
\forall j\ge0,\quad  \Ext_\DX^j(\CM,\cOXY)= \Ext_\DX^j(\CM,\OX)
\end{equation}

We proved in \cite{LME} that the converse is true if $Y$ is a hypersurface. 

Let $d$ be the codimension of $Y$. Then $\BIYX=\CH^d_Y(\OX)$ is the cohomology of $\OX$ with support in $Y$ and 
$\BYX=\CH^d_{[Y]}(\OX)$ is the corresponding algebraic cohomology. 

We proved also that if $\CM$ is regular along $Y$ then 
\begin{equation}\label{equ:byx2}
\forall j\ge0,\quad  \Ext_\DX^j(\CM,\BYX)= \Ext_\DX^j(\CM,\BIYX)
\end{equation}

Assume that $Y$ is a hypersurface and let $j:X\setminus Y\hookrightarrow X$ be the canonical injection.
We denote by $\OX[*Y]$ is the sheaf of meromorphic functions with poles on $Y$ and by $j_*j^{-1}\OX$
the sheaf of holomorphic functions with singularities on $Y$. Then 
$\BIYX=j_*j^{-1}\OX/\OX$ and $\BYX=\OX[*Y]/\OX$.

In \cite{LME}, we proved also that $\CM$ is regular along $Y$ if equality (\ref{equ:byx2}) is
true.

If  $Y$ is not a hypersurface, this is no more true and we have to microlocalize
the definition to get an equivalence.

\section{Microlocal Regularity}

We denote by $\EX$ the sheaf of microdifferential operators of \cite{SKK}, filtered by the order.
We will denote that filtration by $\EX = \bigcup\CE_{X,k}$ and call it the usual filtration.

In \cite{ENS}, we extended the definitions of V-filtrations and $b$-functions
to microdifferential equations and lagrangian subvarieties of the cotangent bundle.
These definitions are invariant under quantized canonical transformations.

Let $\gL$ be a lagrangian conic submanifold of the cotangent bundle $\TX$ and $\CM_\gL$ be a simple holonomic 
$\EX$-module supported by $\gL$. By definition, such a module is generated by a non degenerate section $u_\gL$, that is 
such that the ideal of the principal symbols of the microdifferential operators annihilating $u_\gL$ is the ideal of definition of $\gL$. It always exists locally \cite{SKK}.

Let $\CM_{\gL,k}=\CE_{X,k}u_\gL$. Then the V-filtration on $\EX$ along $\gL$ is defined by:
\begin{equation}
V_k\EX=\ensemble{P\in\EX}{\forall \ell\in \Z, P\CM_{\gL,\ell}\subset \CM_{\gL,\ell+k}}
\end{equation}
This filtration is independent of the choices of $\CM_\gL$ and $u_\gL$, so it is globally defined.

Let $\CO_\gL[k]$ be the sheaf of holomorphic functions on $\gL$ homogeneous of degree $k$ in the
fibers of $\gL\to X$ and $\CO_{(\gL)}=\bigoplus_{k\in\Z}\CO_\gL[k]$. 
Then there is an isomorphism between $\CM_{\gL,k}/\CM_{\gL,k-1}$ and $\CO_{\gL,k}$. By this isomorphism
the graded ring $gr^V\EX$ acts on $\CO_{(\gL)}$ and may be identified to the sheaf $\CaD_{(\gL)}$ of 
differential operators on $\gL$ with coefficients in $\CO_{(\gL)}$.

All these definitions are invariant under quantized canonical transformations \cite{ENS}.

\begin{defin}\label{def:micro}
The holonomic $\EX$-module $\CM$ is \textsl{regular along} $\gL$ (on an open set of $\TX$)
if any section $u$
of $\CM$ is annihilated by a microdifferential operator of the form $\gvt^N + P + Q$
where $P$ is in $\CE_{X,N-1}\cap V_{0}\EX$
and $Q$ is in $\CE_{X,N}\cap V_{-1}\EX$
\end{defin}

We have the fundamental result:
\begin{thm}(theorem 2.4.2 of \cite{LME})\label{thm:regular}
The holonomic $\EX$-module $\CM$ is regular along the conormal $\TYX$ to a submanifold $Y$
of $X$ on the open set $\gO$ of $\TX$ if and only if
\begin{equation}\label{equ:microsol2}
\forall j\ge0,\quad  \Ext_\EX^j(\CM,\CYX)|_\gO= \Ext_\EX^j(\CM,\CIYX)|_\gO
\end{equation}
where $\CYX$ is the sheaf of holomorphic microfunctions of \cite{SKK}.
\end{thm}

As the restriction to the zero section of $\TX$ of the sheaf $\CYX$ is $\BYX$
while the restriction of $\CIYX$ is $\BIYX$, the equation (\ref{equ:byx2})
is a special case of (\ref{equ:microsol2}).

If $\gL$ is not the conormal bundle to a submanifold $Y$, the result is still true 
if $\CYX$ is replaced by a simple holonomic module $\CM_\gL$
\begin{equation}\label{equ:microsol}
\forall j\ge0,\quad  \Ext_\EX^j(\CM,\CM_\gL)= \Ext_\EX^j(\CM,\EIX\ox_\EX\CM_\gL)
\end{equation}

The regularity property is also equivalent to other properties like the existence
of a regular $b$-function or conditions on the microcharacteristic varieties (see \cite{BFUNCT}).

If $\gL$ is a smooth 
part of an irreducible component of the characteristic variety of $\CM$, definition \ref{def:micro} is equivalent to what Kashiwara-Kawa{\"\i} call 
"to have regular singularities along $\gL$" (definition 1.1.11. of \cite{KKREG}) The equivalence is proved in theorem 3.1.7. of \cite{THESE}.

So definition 1.1.16. of \cite{KKREG} may be reformulated as:

\begin{defin}\label{def:regular}
The holonomic $\EX$-module $\CM$ has {\sl regular singularities} (or is {\sl regular}) if for each
irreducible component $\gL$ of its characteristic variety, $\CM$ is regular along a Zarisky open subset
of the regular part of $\gL$. 

A holonomic $\DX$-module is regular if $\EX\ox_\DX\CM$ is regular.
\end{defin}

Let us now recall two important results of Kashiwara-Kawa{\"\i}:

\begin{thm}(theorem 4.1.1. of \cite{KKHOUCH})\label{thm:KK1}
If $\CM$ is a regular holonomic $\EX$-module then it is regular
along any lagrangian submanifold of $\TX$.
\end{thm}

\begin{thm}(theorem 6.4.1 of \cite{KKREG})\label{thm:KK2}
A holonomic $\DX$-module $\CM$  is regular on $X$ if and only if at each point $x\in X$ 
$$\RHOM_\DX(\CM,\OX)_x\simeq \RHOM_\DX(\CM,\widehat\CO_x)$$
\end{thm}
Here $\OX$ is the sheaf of holomorphic functions while $\widehat\CO_x$ is the set
of formal power series at $x$.

\section{A complex Cauchy problem for $\EX$-modules}

Let $X$ be a complex analytic manifold and $Y$ be a smooth hypersurface of $X$.
The inverse image of a $\EX$-module on $Y$ has been defined in \cite{SKK} and
we refer to \cite{SCHAPBOOK} for the details.

Let $\varpi:\TX\x_XY\to\TX$, $\gro:\TX\x_XY\to\TY$ and $\gp:\TX\x_XY\to Y$ be the canonical maps.
The sheaf $\EYX$ may be defined as
$$\EYX=\gp^{-1}\OY\otimes_{\gp^{-1}\OX}\varpi^{-1} \EX$$
If $t$ is an equation of $Y$ we have $\EYX=\EX/t\EX$.

The inverse image of a coherent left $\EX$-module $\CM$ by $i:Y\hookrightarrow X$ is defined as
$$i^*\CM = \CM_Y = \gro_*\left(\EYX\otimes_{\varpi^{-1}\EX}\varpi^{-1}\CM\right)$$

The characteristic variety of a $\EX$-module is, by definition, its support.
A submanifold $Y$ of $X$ is non characteristic for a $\EX$-module $\CM$ if the map
$\gro:\varpi^{-1}\CHM\to \TY$ is proper and finite.

Let $Z$ be a smooth hypersurface of $Y$ and $\gL=\TZX$ be the conormal bundle to $Z$.
Let $(r,s)$ be two rational numbers such that $+\infty>\ge r\ge s\ge 1$.
The microcharacteristic variety of type $(r,s)$ of $\CM$ along $\gL$ has been defined
in \cite{THESE} and \cite{ENS}. It is a subvariety of
$T^*\gL$ denoted by $\CHMrs$. Thanks to lemma \ref{lem:nmicro}, we will not need the definition
of this variety.

The map $\varpi$ induces an isomorphism on $\gL$ while $\gro$ induces a map $\gro:\gL=\TZX\to\gL'=\TZY$.
The submanifold $Y$ is non microcharacteristic of type $(r,s)$ for $\CM$ if $\CHMrs$
does not meet the conormal bundle $T^*_\gL\gL'$ outside of the zero section.

\begin{lem}\label{lem:nmicro}
The conormal bundle $T^*_\gL\gL'$ is contained in the zero section hence $Z$ is non non microcharacteristic of type $(r,s)$
for any $(r,s)$ and any coherent $\EX$-module.
\end{lem}

\begin{proof}
The problem being local, take local coordinates $(x,y,t)$ of $X$ such that 
$Y=\ensemble{(x,y,t)\in X}{t=0}$ and $Z=\ensemble{(x,y,t)\in X}{t=0, y=0}$. Let $(x,y,t,\gx,\gh,\gt)$
be the local coordinates of $\TX$ defined by $(x,y,t)$.

Then $\gro$ is given by $\gro(x,y,0,\gx,\gh,\gt)=(x,y,\gx,\gh)$ and 
$\gL = \ensemble{(x,y,t,\gx,\gh,\gt)\in \TX}{t=0, y=0, \gx=0}$. 
So $\gL'=\gro(\gL)= \ensemble{(x,y,\gx,\gh)\in \TY}{ y=0, \gx=0}$
and the map $\gro:\TL\to\TL'$ is a submersion.

Thus the conormal bundle $T^*_\gL\gL'$ is contained in the zero section and 
$Z$ is non non microcharacteristic for any microcharacteristic variety.
\end{proof}

\begin{prop}\label{prop:cauchy}
Let $\CM$ be  holonomic $\EX$-module which is defined in a neighborhood $\gO$ of $\TZX$
and assume that $Y$ is non characteristic for $\CM$. Then we have:
\begin{align*}
\gro_*\RHOM_\EX(\CM,\CZX)&\xrightarrow{\ \sim\ }\RHOM_\EY(\CM_Y,\CZY)\\
\gro_*\RHOM_\EX\CM,\CIZX)&\xrightarrow{\ \sim\ }\RHOM_\EY(\CM_Y,\CIZY)
\end{align*}
\end{prop}

\begin{proof}
By lemma \ref{lem:nmicro}, $Y$ is non microcharacteristic of type $(\infty,1)$ and $(1,1)$
for $\CM$. Hence we may apply directly \cite[Prop. 3.2.2.]{THESE} which give the result 
with $\CZX$ in case $(\infty,1)$  and $\CIZX$ in case $(1,1)$.
Remark that the latter case had been proved before by Kashiwara-Schapira \cite{KSI}.
\end{proof}
The result is completely different if $Z$ is not contained in $Y$ \cite{LME}.

\section{Microlocal regularity}

 The aim of this section is to prove
a microlocal version of theorem \ref{thm:KK2}.

Let $M$ be a conic symplectic manifold. Here we assume that $M$ is a complex manifold but
proposition \ref{prop:fol} is true as well in the real differentiable case.

We recall that a structure of conic symplectic manifold on a complex manifold $M$
is given by a $1$-form $\go_M$ whose differential $\gs_M=d\go_M$ is a symplectic $2$-form
on $M$. A {\sl conic lagrangian foliation} of $M$ is a foliation by conic lagrangian submanifolds.

\begin{prop}\label{prop:fol}
Let $M$ be conic symplectic manifold with a conic lagrangian foliation. There is locally
a homogeneous symplectic map from $M$ to the cotangent bundle $\TX$
of a complex manifold $X$ which transforms the leaves into the fibers of $\gp:\TX\to X$.
\end{prop}

This result has been proved in the non-homogeneous case and when $M$ is a Banach space by
Weinstein \cite[Cor. 7.2]{WEIN}.
\begin{proof}
A lagrangian variety is involutive, hence if two functions vanish on a lagrangian variety $\gL$,
their Poisson bracket vanishes on $\gL$. But the Poisson bracket depend only on the derivative of the functions
hence if two functions are constant on $\gL$, their Poisson bracket vanishes on $\gL$.

Let $\CO$ be the sheaf of ring of holomorphic functions on $M$ which are constant on the leaves
of the foliation. The Poisson bracket of two functions of $\CO$ vanishes everywhere.
Let $2n$ be the dimension of $M$. As the foliation has codimension $n$, we can find
locally $n$ functions $u_1,\dots,u_n$ in $\CO$ whose differentials are linearly independent at each point.
The Poisson bracket of two of them is always $0$ hence by the proof of 
Darboux theorem as it is given in \cite[theorem 3.5.6.]{DUIS},
there are $n$ functions $v_1,\dots,v_n$ on $M$ such that $u_1,\dots,u_n,v_1,\dots,v_n$ is a canonical
symplectic (non-homogeneous) system of coordinates for $M$. By definition, the canonical symplectic $2$-form
$\gs_M$ of $M$ is thus  equal to $\sum dv_i\wedge du_i$. 

The functions $u_1,\dots,u_n,$ are constant on the leaves of the foliation which are conic varieties,
hence they are constant on the fibers of the $\C$-action.
Let $\gth$ be Euler vector field associated to this action and for $i=1,\dots,n$ let $w_i=\gth(v_i)$.
As $\gth(u_i)=0$, $\gth$ is equal to $\sum_{i=1}^n w_i(u,v)\dsur{v_i}$.

By definition the canonical $1$-form $\ga$ of the homogeneous symplectic manifold $M$ is equal to
the inner product $\go\rfloor\gth$, hence $\ga=\sum w_i du_i$. As $d\ga=\go$, the functions
$u_1,\dots,u_n,$ $w_1,\dots,w_n$ define a a canonical
symplectic homogeneous system of coordinates for $M$.

In the coordinates $(u,w)$, the leaves are given by $u=constant$, which shows the proposition.
\end{proof}

\begin{thm}\label{thm:microreg}
Let $\gO$ be a conic open subset of $\TX$ and $(\gL_\ga)$ be a conic lagrangian 
foliation of $\gO$.

Let $\CM$ be a holonomic $\EX$-module defined on $\gO$. If $\CM$ is regular along 
each leaf  $\gL_\ga$,  then $\CM$ is a regular holonomic module on $\gO$.
\end{thm}

\begin{proof}
We will prove the proposition by induction on the dimension of $X$. As definition \ref{def:regular}
concerns only the part of the characteristic variety outside of the zero section of $\TX$,
we will work in a neighborhood of a point of $\dTX$. 

If the dimension of $X$ is $1$, the characteristic variety of $\CM$ is 
the union of the conormal bundles to isolated points of $X$. On
the other hand, a conic lagrangian foliation is necessarily given by the conormal bundles
to the points of $X$. Hence by the hypothesis, $\CM$ is regular along each components
of its characteristic variety, hence 
$\CM$ is regular by \cite[Def 1.1.16.]{KKREG}.

Assume now that the dimension of $X$ is $>1$. The problem being local on $\TX$, we may use 
proposition \ref{prop:fol} and transform the foliation into the union of
the conormals to the points of an open subset $U$ of $X$. So $\gO$ is a conic open
subset of $\gp^{-1}(U)$ and $\CM$ is regular along $T^*_{\{x\}}X$ on $\gO$ for $x\in U$.

Let $Y$ be a smooth hypersurface of $X$ which is 
non characteristic for $\CM$. Let $x$ be a point of $U \cap Y$. 
Then by lemma \ref{prop:cauchy} applied to $Z=\{x\}$ we have 
\begin{align*}
\RHOM_\EY(\CM_Y,\CZY)&\xrightarrow{\ \sim\ }\gro_*\RHOM_\EX(\CM,\CZX)\\
&\xrightarrow{\ \sim\ }\gro_*\RHOM_\EX\CM,\CIZX)\xrightarrow{\ \sim\ }\RHOM_\EY(\CM_Y,\CIZY)
\end{align*}
So $\CM_Y$ is regular along $T^*_{\{x\}}Y$ on $\gO\cap \TY$ for $x\in U\cap Y$ , hence by the hypothesis of induction,
$\CM_Y$ is regular on $\gO\cap \TY$.

According to definition \ref{def:regular}, we will now prove that $\CM$ is regular by proving that 
it is regular along a Zarisky open subset of each irreducible component of its characteristic variety. 
Such an irreducible component is a conic lagrangian subvariety of $\TX$ hence generically
we may assume that it is the conormal to a smooth subvariety $Z$ of $X$.
Let $\gL$ be such a component. If $\gL$ is the conormal to a point of $X$, then $\CM$ is regular along
$\gL$ by the hypothesis. So we may assume that $\gL$ is not the conormal to a point and 
consider a point $\gm$ of $\TX$ where $\gL$ is the conormal to a submanifold $Z$ of $X$ of dimension $\ge 1$.

Locally, there are local coordinates $(x_1\dots,x_p,t_1,\dots,t_q)$ of $X$
such that $\gm$ has coordinates $x=0$, $t=0$, $\gx=0$, $\gt=(1,0,\dots,0)$ and  $\gL=\ensemble{(x,t,\gx,\gt)\in\TX}{x=0,\gt=0}$. 
Let $V_a=\ensemble{(x,t)\in X}{t_1=a}$.
The conormal to the hypersurface $V_a$ does not meet $\gL$ hence $V_a$ is not characteristic for $\CM$ and
$\CM_{V_a}$ is regular if $a<<1$. Now we apply theorem 6.4.5. of \cite{KKREG} and get that $\CM$ is regular near $\gm$
hence regular along $\gL$.
\end{proof}

\section{Applications and examples}

Let $\gS$ be a submanifold of $\TX$ with a conic lagrangian foliation. This implies that 
$\gS$ is conic involutive \cite[th 3.6.2]{DUIS}, We assume that this foliation
is the restriction of a conic lagrangian foliation of $\TX$.

Let $\CM$ be a holonomic $\EX$-module whose characteristic variety is contained
in $\gS$ and which is regular along each leaf of this foliation. Then by theorem
\ref{thm:microreg}, $\CM$ is regular hence regular along any lagrangian submanifold 
of $\TX$.

A typical example of this situation is given by "maximally degenerated" involutive manifolds,
as we will explain now.

The canonical projection $\TX\to X$ defines a map $\TX\x_X\TX\to T^*(\TX)$ which
composed with the diagonal map $\TX\to\TX\x_X\TX$ defines the canonical 1-form of $\TX$ that is
$\go_X:\TX\to T^*(\TX)$. We now restrict our attention to the complementary $\dTX$
of the zero section in $\TX$.
The set of points of $\gS\cap\dTX$ where $\go_X|_\gS$ vanishes is isotropic 
hence of dimension less or equal
to the dimension of $X$. It is called the {\sl degeneracy locus} of $\gS$.

\begin{defin}
The involutive submanifold $\gS$ of $\dTX$ is said to be {\sl maximally degenerated} if
the degeneracy locus is of maximal dimension that is the dimension of $X$.
\end{defin}

Then the degeneracy locus is a lagrangian submanifold $\gL_0$ of $\dTX$. Being
involutive, the manifold $\gS$ has a canonical foliation by bicharacteristic leaves.

\begin{lem}
Let $\gS$ be a maximally degenerated involutive submanifold of $\dTX$ and $\gL_0$ its degeneracy locus.
\begin{enumerate}
\item $\gL_0$ is a union of bicharacteristic leaves.
\item For each leaf $L$ of $\gL_0$, there is one and only one lagrangian homogeneous submanifold of
$\gS$ whose intersection with $\gL_0$ is exactly $L$.
\item Theses lagrangian submanifolds define a foliation of $\gS$ which we will call the "lagrangian foliation".
\end{enumerate}
\end{lem}

For a detailed study of maximally degenerated involutive submanifold, we refer to Duistermaat \cite{DUIS}.

Locally on $\gL_0$, there is a homogeneous symplectic transformation of $\dTX$ and local coordinates of $X$
$(x_1,\dots,x_{n-p},t_1,\dots,t_p)$ which
transform $\gS$ into $\ensemble{(x,t,\gx,\gt)\in \dTX}{t=0}$. Then $\gL_0=\ensemble{(x,t,\gx,\gt)\in \dTX}{t=0, \gx=0}$
and the lemma is easy to prove (see example \ref{ex:ploit}).

We may now apply theorem \ref{thm:microreg} to these involutive manifold: 

\begin{cor}\label{cor:microreg}
Let $\gS$ be a maximally degenerated submanifold of $\dTX$ with degeneracy locus $\gL_0$ and lagrangian foliation
$(L_\ga)$.

Let $\CM$ be a holonomic $\EX$-module whose characteristic variety is contained in $\gS$ and which
is regular along each leaf $L_\ga$. Then $\CM$ is regular holonomic hence regular along $\gL_0$.
\end{cor}

\begin{ex}\label{ex:ploit}

Let $Y$ be a submanifold of $X$ and $\gS=\gp^{-1}(Y)$ where $\gp:\dTX\to X$ is the canonical projection.
Then the degeneracy locus of $\gS$ is $\TYX$, the conormal bundle to $Y$.  The lagrangian foliation
is given by the conormal bundles to the points of $Y$.

In local coordinates $(x,t)$ of $X$ where $Y$ is given by $t=0$, the bicharacteristic leaves of $\gS$
are the sets
$$F_{x_0,\gx_0}=\ensemble{(x,t,\gx,\gt)\in \dTX}{t=0, x=x_0, \gx=\gx_0}$$
The degeneracy locus is $\gL_0=\ensemble{(x,t,\gx,\gt)\in \TX}{t=0, \gx=0, \gt\ne 0}$ while the lagrangian foliation is given
by the manifolds
$$\gF_{x_0}=\ensemble{(x,t,\gx,\gt)\in \TX}{t=0, x=x_0, \gt\ne 0}$$
\end{ex}

\begin{ex}
Let $X=\C^n$ and $\gS=\ensemble{(x,\gx)\in \TX}{\scal x \gx = 0, \gx\ne 0}$.

Then the degeneracy locus of $\gS$ is the conormal bundle to the origin of $\C^n$ and  the lagrangian foliation
is given by the conormal bundles to the hyperplanes of $X$ which contain the origin.
\end{ex}

\begin{ex}
More generally, we may consider a linear subvariety $Z$ of $X$ and $\gS$  the union of the conormal bundles to
the hyperplanes which contain $Z$. The lagrangian foliation is given by these conormals and the degeneracy locus
is the conormal bundle to $Z$.
\end{ex} 

\begin{ex}
Let $V=\ensemble{(x,t,\gx,\gt)\in T^*\C^2}{x\gx+t\gt=\gx^2/\gt}$. Then $\gL_0$ is the conormal to
the curve $S= \{t+x^2=0\}$ while the lagrangian foliation is given by the conormal to the tangent
lines to $S$.
\end{ex} 

Applied to $\DX$-modules, corollary \ref{cor:microreg} gives the following result:
\begin{cor}\label{cor:reg}
Let $\gS$ be a maximally degenerated submanifold of $\dTX$ with degeneracy locus $\gL_0$ and lagrangian foliation
$(L_\ga)$.

Let $\CM$ be a holonomic $\DX$-module. We assume that the characteristic variety of $\CM$ is contained in $\gS$
in a conic neighborhood $\gO$ of $\gL_0$ in $\dTX$. If $\CM$ is regular along each leaf $L_\ga$
in $\gO$, then $\CM$ is regular along $\gL_0$.
\end{cor}

Remark that in this corollary, $\CM$ is microlocally regular in a neighborhood of $\TYX\cap \dTX$ but
may be not regular as a $\DX$-module.

\begin{ex}
The singular support of a $\DX$-module is the projection of the intersection of its characteristic variety
with $\dTX$. It is also the set of points of $X$ where $\CM$ is not locally isomorphic to a power of $\OX$.
Assume that the singular support of a $\DX$-module $\CM$ is a normal crossing divisor 
and that $Y$ is an irreducible component of this divisor. Then, if $\CM$
is regular along each point of $Y$, it is regular along $Y$.

More generally, if no component of the characteristic variety is the conormal to a variety tangent to $Y$
or singular on $Y$, then the result is still true because the characteristic variety is contained
in $\gp^{-1}(Y)$ in a neighborhood of $\TYX$.
\end{ex}

\section{A counterexample}\label{sect:cex}

In this section, we give an example which shows that the condition
on the characteristic variety in corollary \ref{cor:microreg} is necessary. In particular,
if the singular support of a holonomic $\DX$-module $\CM$ has components tangent
to the manifold $Y$, then $\CM$ may be regular along each point of $Y$ but not regular
along $Y$.

Let $X=\C^2$ with coordinates $(x,t)$, we denote
by $D_x=\dsur{x}$ and $D_t=\dsur{t}$ the corresponding derivations. We consider
two differential operators $P=x^2D_x+1$ and $Q=t$.

Let $\CI$ be the ideal of $\DX$ generated by $P$ and $Q$ and $\CM$ be the holonomic 
$\DX$-module $\DX/\CI$.

We denote by $Y_0$ the hypersurface of equation $\{t=0\}$ and by $Y_\gf$ the hypersurface
of equation $\{t-\gf(x)=0\}$ where $\gf$ is a holomorphic function of one variable
defined in a neighborhood of $0$ in $\C$.

\begin{prop} \label{prop:enum}
\
\begin{enumerate}
\item The module $\CM$ is regular along $Y_0$.
\item If $\gf(0)=0$ and $\gf(x)\ne 0$ when $x\ne 0$, the module $\CM$ is not regular along $Y_\gf$.
\item The module $\CM$ is not regular along $\{0\}$. In fact, it is not regular along $T^*_{\{0\}}X$
at any point of $T^*_{\{0\}}X$.
\end{enumerate}
\end{prop}

\begin{proof}
As the module is supported by $Y_0$ it is trivially regular along $Y_0$.

To prove that $\CM$ is not regular along $Y_\gf$, we will prove that 
$$\Ext^1_\DX(\CM,\BIYfiX/\BYfiX)_0\ne 0$$.

Let $j:X\setminus Y_\gf\hookrightarrow X$ be the canonical injection. Let $j_*j^{-1}\OX$
be the sheaf of holomorphic functions with singularities on $Y_\gf$ and $\OX[*Y_\gf]$
be the sheaf of meromorphic functions with poles on $Y_\gf$. 
By definition 
$\BIYfiX=j_*j^{-1}\OX/\OX$ and $\BYfiX=\OX[*Y_\gf]/\OX$ hence we have to calculate the $\Ext^1$
with values in $\CF=j_*j^{-1}\OX/\OX[*Y_\gf]$. 
As $P$ and $Q$ commute the module $\CM$ admits a free resolution
$$0\xrightarrow{\ \ }{(\DX)^2}\xrightarrow{\ \binom P Q\ }{(\DX^2)}\xrightarrow{\ (P,Q)\ }{\DX}\xrightarrow{\ \ }{\CM} \xrightarrow{\ \ }0$$
Hence $\Ext^1_\DX(\CM,\CF)$ vanishes if and only if the system
\begin{equation}\label{equ:nosol}
\left\{
\aligned
Pf&=g\\
Qf&=h
\endaligned
\right.
\qquad\textrm{with}\qquad Ph=Qg
\end{equation}
has a solution $f$ in $\CF$ for any data $(g,h)$ in $\CF\x\CF$.

\begin{lem}\label{lem:exp}
There exists a function $h(x,t)$ in $j_*j^{-1}\OX$ such that $h(x,0)=exp(1/x)$.
\end{lem}
\begin{proof}
There is some integer $n>0$ and a function $\gy$ such that $\gf(x)=x^n\gy(x)$ and $\gy(0)\ne 0$.
Let $\gl_0(x)=1$ and for $j\ge 1$, let $\gl_j(x)=\sum_{k=0}^{n-1}\frac{x^k}{(nj-k)!}$. We have 
$\gl_0(x)+\sum_{j\ge 1}\gl_j(x)\frac 1 {x^{nj}}=exp(1/x)$. The function
$$h(x,t)=\sum_{j\ge 0} \gy(x)^j \gl_j(x)\frac 1 {(\gf(x)-t)^j}$$
is a solution to the lemma.
\end{proof}

{\sl Proof of proposition continued:}
Let $h(x,t)$ be the function given by lemma \ref{lem:exp}. We have 
$$\left(P(x,D_x)h(x,t)\right)|_{t=0}=P(x,D_x)h(x,0)=P(x,D_x)exp(1/x)=0$$
hence there exists some function $g$ in $j_*j^{-1}\OX$ such that $tg(x,t)=P(x,D_x)h(x,t)$.

As $h(x,0)$ is a function on $\C\setminus\{0\}$ which is not meromorphic the equation $Qf=h$
has no solution in $\CF=j_*j^{-1}\OX/\OX[*Y_\gf]$. Hence equation \ref{equ:nosol} has no solution and 
 $\Ext^1_\DX(\CM,\CF)$ does not vanish. This shows point 2) of the proposition.
 
To prove point 3) we consider the function $f(x,t)=(1/t)exp(1/x)$ which defines the element
$u=\sum_{j\ge 0}\frac{(-1)^j}{j!(j+1)!}\gd^{(j)}(x)\gd(t)$ in $\BIOX$. As $u$ is a solution
of $Pu=Qu=0$ which do not belong to $\BOX$, the module $\CM$ is not regular on $\{0\}$.

Moreover, we may consider the microfunction $v$ in $\CIOX$ of symbol $\sum_{j\ge 0}\frac{(-1)^j}{j!(j+1)!}\gx^j$,
that is the microfunction associated to $u$. It is a solution of $Pv=Qv=0$ which do not belong
to $\COX$ at any point of $T^*_{\{0\}}X$. This shows point 3) of the proposition.
\end{proof}

Let $Z=\C^2$ with coordinates $(y,s)$. The coordinates will be $(x,t,\gx,\gt)$ on $\TX$ and  $(y,s,\gh,\gs)$ on $\TZ$.  
The Legendre transform is defined from $\TX$ to $\TZ$ when $\gt\ne 0$ and $\gs \ne 0$ by the equations
\begin{equation*}
\left\{
\aligned
y&=\gx\gt^{-1}\\
s&=t+x\gx\gt^{-1}\\
\gh&=-x\gt\\
\gs&=\gt
\endaligned
\right.
\end{equation*}
According to \cite[\S 3.3 ch.II]{SKK}, a quantized canonical transformation associated to it is given by:
\begin{equation*}
\left\{
\aligned
x&=-D_yD_s^{-1}\\
t&=(D_ss+D_yy)D_s^{-1}=s+yD_yD_s^{-1}+2D_s^{-1}\\
D_x&=yD_s\\
D_t&=D_s
\endaligned
\right.
\end{equation*}
So, if we apply this transformation to the $\EX$-module $\EX\ox_{\gp^{-1}\DX}\gp^{-1}\CM$
with $\gp:\TX\to X$, we get the coherent $\EZ$-module $\widetilde\CN=\EZ/(\EZ\widetilde P+ \EZ\widetilde Q)$
with 
\begin{equation*}
\left\{
\aligned
\widetilde P(x,s,D_x,D_s)&=D^2_yD_s^{-2}yD_s+1=yD^2_yD_s^{-1}+2D_yD_s^{-1}+1\\
\widetilde Q(x,s,D_x,D_s)&=s+yD_yD_s^{-1}+2D_s^{-1}
\endaligned
\right.
\end{equation*}

Consider the following Lagrangian manifolds:
\begin{alignat*}{4}
&\gL_0&&=T^*_{Y_0}X&&=\ensemble{(x,t,\gx,\gt)\in\TX}{t=0,\gx=0}&&\\
&\gL_\ga&&=T^*_{Y_\gf}X&&=\ensemble{(x,t,\gx,\gt)\in\TX}{t=\ga x,\gx=-\ga\gt}&&\textrm{with } \gf(x)=\ga x\\
&\gL'_\gl&&=T^*_{Y_\gf}X&&=\ensemble{(x,t,\gx,\gt)\in\TX}{t=4\gl x^2,\gx=-8\gl x\gt}\quad&&\textrm{with } \gf(x)=4\gl x^2\\
&\gL'_0&&=T^*_{\{0\}}X&&=\ensemble{(x,t,\gx,\gt)\in\TX}{t=0,x=0}&&\\
\\
&\widetilde\gL_0&&=T^*_{\{0\}}Z&&=\ensemble{(s,y,\gh,\gs)\in\TZ}{s=0,y=0}&&\\
&\widetilde\gL_\ga&&=T^*_{\{p_\ga\}}Z&&=\ensemble{(s,y,\gh,\gs)\in\TZ}{s=0,y=\ga}\quad&&\\
&\widetilde\gL_\gl&&=T^*_{Y_\gf}Z&&=\ensemble{(s,y,\gh,\gs)\in\TY}{s+y^2/\gl=0, \gh=2(y/\gl)\gs}
  \quad&&\textrm{with } \gf(y)=-y^2/\gl\\
&\widetilde\gL'_0&&=T^*_{Y_0}Z&&=\ensemble{(s,y,\gh,\gs)\in\TZ}{s=0,\gh=0}&&
\end{alignat*}

By the Legendre transform, the manifolds $\gL_0$, $\gL_\ga$, $\gL'_\gl$ and $\gL'_0$ are respectively 
transformed into $\widetilde\gL_0$, $\widetilde\gL_\ga$, $\widetilde\gL'_\gl$ and $\widetilde\gL'_0$
outside of the set $\{\gt=0\}$. As the regularity along a lagrangian manifold is invariant under 
quantized canonical transformation, we get from proposition \ref{prop:enum} that $\widetilde\CN$ is a holonomic
$\EZ$-module which is regular along $\widetilde\gL_0$ but not regular along
$\widetilde\gL_\ga$, $\widetilde\gL'_\gl$ and $\widetilde\gL'_0$.

Let again be  $Z=\C^2$ with coordinates $(y,s)$ and let us define the following differential operators:
\begin{equation*}
\left\{
\aligned
P_1(x,s,D_x,D_s)&=yD^2_y+2D_y+D_s\\
Q_1(x,s,D_x,D_s)&=sD_s+yD_y+3
\endaligned
\right.
\end{equation*}
Let $\CJ$ be the ideal of $\DZ$ generated by $P_1$ and $Q_1$ and $\CN$ be the holonomic 
$\DZ$-module $\DZ/\CJ$.

\begin{prop}\label{prop:mainex}
The module $\CN$ is regular along all points of $Y=\ensemble{(y,s}{s=y^2}$ but is not
regular along $Y$ itself.
\end{prop}

\begin{rem}
We have also that $\CN$ is irregular along each point of $Y_0=\ensemble{(y,s}{s=0}$ except $0$
and is irregular along $Y_0$.
\end{rem}

\begin{proof}
The operator $Q_1$ is a $b$-function for $\CN$ at $0$ hence $\CN$ is regular along this point (locally but also
microlocally at each point of $T^*_{0}Z$). Moreover, the characteristic variety of $\CN$ is contained in
$\ensemble{(s,y,\gh,\gs)\in\TZ}{s\gs=0,y\gh=0}$ hence $\CN$ is elliptic at each point of $Y$ except $0$, so
it is regular along these points.

The module $\EZ\ox_{\gp^{-1}\DZ}\gp^{-1}\CN$ is by definition isomorphic to $\widetilde\CN$ outside of 
the set $\{\gt=0\}$. We have seen that $\widetilde\CN$ is not regular along $\widetilde\gL'_\gl$. 
Hence, taking $\gl=-1$, we get that $\CN$ is not regular
along $Y$ as well as $Y_0$ and its non zero points.
\end{proof}

\providecommand{\bysame}{\leavevmode\hbox to3em{\hrulefill}\thinspace}
\providecommand{\MR}{\relax\ifhmode\unskip\space\fi MR }
\providecommand{\MRhref}[2]{%
  \href{http://www.ams.org/mathscinet-getitem?mr=#1}{#2}
}
\providecommand{\href}[2]{#2}

\vespa\vespa
\begin{flushright}
  \begin{minipage}[t]{8cm}

Universit\'e de Grenoble

   Institut Fourier

UMR 5584  CNRS/UJF

   BP 74

   38402 St Martin d'H\`eres Cedex

{\bf email:} Yves.Laurent@ujf-grenoble.fr

  \end{minipage}
\end{flushright}


\begin{thebibliography}{10}

\bibitem{GRANGER}
F.~Castro-Jimenez and M.~Granger, \emph{Gevrey expansion of hypergeometric
  integrals {I}}, arXiv:1212.1410v2 (2013).

\bibitem{DUIS}
J.J. Duistermaat, \emph{Fourier integral operators}, Courant Institute of
  Mathematical Sciences, New York University, 1973.

\bibitem{KVAN}
M.~Kashiwara, \emph{Vanishing cycles and holonomic systems of differential
  equations}, Lect. Notes in Math., vol. 1016, Springer, 1983, pp.~134--142.

\bibitem{KKHOUCH}
M.~Kashiwara and T.~Kawa{\"\i}, \emph{Second microlocalization and asymptotic
  expansions}, Complex Analysis, Microlocal Calculus and Relativistic Quantum
  Theory, Lect. Notes in Physics, vol. 126, Springer, 1980, pp.~21--76.

\bibitem{KKREG}
\bysame, \emph{On the holonomic systems of microdifferential equations {III}.
  systems with regular singularities}, Publ. RIMS, Kyoto Univ. \textbf{17}
  (1981), 813--979.

\bibitem{KSI}
M.~Kashiwara and P.~Schapira, \emph{Probl{\`e}me de {C}auchy dans le domaine
  complexe}, Inv. Math. \textbf{46} (1978), 17--38.

\bibitem{BFUNCT}
Y.~Laurent, \emph{Regularity and b-functions for {$\mathcal D$}-modules},
arXiv:1501.06932. (2015)

\bibitem{THESE}
\bysame, \emph{Th\'eorie de la deuxi\`eme microlocalisation dans le domaine
  complexe}, Progress in Math., vol.~53, Birkh{\"a}user, 1985.

\bibitem{ENS}
\bysame, \emph{Polygone de {N}ewton et b-fonctions pour les modules
  microdiff\'erentiels}, Ann. Ec. Norm. Sup. 4e s\'erie \textbf{20} (1987),
  391--441.

\bibitem{INV}
\bysame, \emph{Vanishing cycles of {$\mathcal D$}-modules}, Inv. Math.
  \textbf{112} (1993), 491--539.

\bibitem{LME}
Y.~Laurent and Z.~Mebkhout, \emph{Pentes alg\'ebriques et pentes analytiques
  d'un {$\mathcal D$}-module}, Ann. Ec. Norm. Sup. 4e s\'erie \textbf{32}
  (1999), 39--69.

\bibitem{SKK}
M.~Sato, T.~Kawa{\"\i}, and M.~Kashiwara, \emph{Hyperfunctions and
  pseudo-differential equations}, Lect. Notes in Math., vol. 287, Springer,
  1980, pp.~265--529.

\bibitem{SCHAPBOOK}
P.~Schapira, \emph{Microdifferential systems in the complex domain},
  Grundlehren der Math., vol. 269, Springer, 1985.

\bibitem{WEIN}
A.~Weinstein, \emph{Symplectic manifolds and their lagrangian submanifolds},
  Adv. in Math. \textbf{6} (1971), 329--346.

\end{thebibliography}
\enddocument